\documentclass[12pt]{amsart}
\usepackage[DIV10, headinclude, BCOR 10mm]{typearea}
\usepackage{latexsym}
\usepackage{amssymb}
\usepackage{graphicx}
\usepackage{paralist} 

\newtheorem{thm}{Theorem}[section]
\newtheorem{prop}[thm]{Proposition}
\newtheorem{lem}[thm]{Lemma}
\newtheorem{cor}[thm]{Corollary}
\newtheorem{conj}[thm]{Conjecture}

\theoremstyle{remark}
\newtheorem{rem}[thm]{Remark}
\newtheorem{ex}[thm]{Example}

\theoremstyle{definition}
\newtheorem{defn}[thm]{Definition}

\newcommand{\bN}{\mathbb N}
\newcommand{\bZ}{\mathbb Z}

\newcommand{\C}{\mathbb{ C}}
\newcommand{\Q}{\mathbb{ Q}}
\newcommand{\R}{\mathbb{ R}}
\newcommand{\Z}{\mathbb{ Z}}

\makeatletter
\newcommand\norm{\bBigg@{0.8}}

\makeatother
\newcommand{\inparens}[2][flex]{\csname #1l\endcsname(#2%
  \csname #1r\endcsname)\mathclose{}}
\newcommand{\inangles}[2][flex]{\csname #1l\endcsname\langle#2%
  \csname #1r\endcsname\rangle\mathclose{}} 

\DeclareMathOperator{\im}{im}
  
\newcommand{\sv}[2][flex]{\csname #1l\endcsname\|#2%
  \csname #1r\endcsname\|} 
\newcommand{\lone}[2][flex]{\csname #1l\endcsname\|#2%
  \csname #1r\endcsname\|_1} 
\newcommand{\supn}[2][flex]{\csname #1l\endcsname\|#2%
  \csname #1r\endcsname\|_{\infty}} 
 
\newcommand{\hsing}[3][flex]{H_{#2}\inparens[#1]{#3}}

\title[Fundamental classes not representable by products]
      {Fundamental classes\\ not representable by products}
\author{D.~Kotschick}
\address{Mathematisches Institut, {\smaller LMU} M\"unchen,
Theresienstr.~39, 80333~M\"unchen, Germany}
\email{dieter@member.ams.org}
\author{C.~L\"oh}
\address{Mathematisches Institut, {\smaller WWU} M\"unster,
  Einsteinstr.~62, 48143~M\"unster, Germany}
\email{clara.loeh@uni-muenster.de}
\date{\today; \copyright{\ D.~Kotschick and C.~L\"oh 2008};}
\thanks{The first author is a member of the DFG priority program in Global Differential Geometry. 
The second author would like to thank Wolfgang L\"uck for clarifying discussions.} 

  \usepackage{ifpdf}
  \ifpdf
    \usepackage[pdftex,colorlinks,linkcolor=blue,citecolor=blue,urlcolor=blue,a4paper,hypertexnames=true,plainpages=false]{hyperref}
    \hypersetup{pdfauthor={D.~Kotschick, C.~L\"oh},
                pdftitle ={Fundamental classes not representable by products},
                }
    \usepackage[all]{xy}
    \SelectTips{cm}{}
  \else
    \usepackage[colorlinks,linkcolor=blue,citecolor=blue,urlcolor=blue,a4paper,hypertexnames=true,plainpages=false]{hyperref}
    \usepackage[all, dvips]{xy}
    \SelectTips{cm}{}
  \fi

\begin{document}

\begin{abstract}
  We prove that rationally essential manifolds with suitably large
  fundamental groups do not admit any maps of non-zero degree from
  products of closed manifolds of positive dimension. Particular
  examples include all manifolds of non-positive sectional curvature
  of rank one and all irreducible locally symmetric spaces of
  non-compact type.  For closed manifolds from certain classes, say
  non-positively curved ones, or certain surface bundles over
  surfaces, we show that they do admit maps of non-zero degree from
  non-trivial products if and only if they are virtually diffeomorphic
  to products.
\end{abstract}

\maketitle

\section{Introduction}\label{s:intro}

The existence of a continuous map~$M\longrightarrow N$ of non-zero
degree defines an interesting transitive relation, denoted~$M\geq N$,
on the homotopy types of closed oriented
manifolds~\cite{MT,Gromov,CT}. In every dimension, homotopy spheres
represent an absolutely minimal element for this relation. In
dimension two, the relation coincides with the order given by the
genus, and substantial information is now known in dimension three as
well~\cite{WangICM}.

In general, if $M$ dominates $N$, then $M$ is at least as complicated
as $N$. For example, $M\geq N$ implies that the Betti numbers of~$M$
are at least as large as those of~$N$ and that the fundamental group
of~$M$ surjects onto a finite index subgroup of the fundamental group
of~$N$.  However, these necessary conditions are in general very far
from being sufficient, and the relation $\geq$ is poorly understood in
higher dimensions.  Nevertheless, interesting results about it have
been obtained for two different kinds of targets $N$: either $N$ is
assumed to be highly connected, or $N$ is assumed to be negatively
curved, or at least to have a large universal covering in a suitable
sense.  In the highly connected case the methods of algebraic topology
have been successfully applied to the study of the domination relation
in the work of Duan and Wang~\cite{DW1,DW2}.  At the opposite end of
the spectrum, for manifolds with large universal coverings,
interesting results have been obtained via more geometric methods.
These include Gromov's theory of bounded
cohomology~\cite{Gromov,ivanov}, most notably the concept of
simplicial volume, and the theory of harmonic maps, as applied to the
domination question by Siu, Sampson, Carlson--Toledo and others,
cf.~\cite{CT,ABCKT} and the literature quoted there.

In this paper we prove that certain manifolds cannot be dominated by
any non-trivial product of closed manifolds.  One of our motivations
stems from Gromov's discussion of functorial semi-norms on
homology~\cite[Chapter~5$G_+$]{gromovmetric}, where the issue of
representing even degree homology classes by products of surfaces is
raised. Gromov suggested that many interesting homology classes should
not be representable by products (of surfaces) and singled out the
fundamental classes of irreducible locally symmetric spaces as
specific candidates.  As a special case of our results, we confirm
Gromov's suggestion in the following general form: if $P$ is any
non-trivial product of closed manifolds and $N$ is a closed
irreducible locally symmetric space of non-compact type, then~$P\ngeq
N$; see Corollary~\ref{c:locsym}. Another motivation for results of
the type~$P \ngeq N$ comes from the study of diffeomorphism groups,
where the special case~$(M\times S^1) \ngeq N$ occurs~\cite{KKM}.

Our methods, while inspired by differential geometry and by Gromov's
theory of the simplicial volume~\cite{Gromov}, are, for the most part,
elementary.  We combine basic homotopy theory with the discussion of
certain purely algebraic properties of fundamental groups.  More
specifically, we translate domination by products on the level of
manifolds into properties of the corresponding fundamental groups. As
the images of the fundamental groups of the factors commute in the
fundamental group of the target and generate a subgroup of finite
index, domination by a product forces the fundamental group of the
target to have a certain amount of commutativity. This alone is often
enough to prove $P\ngeq N$.

For the formulation of our results it is convenient to use the
following terminology due to Gromov~\cite{Gromov}, compare
also~\cite{HKRS}.
\begin{defn}
  A closed, oriented, connected $n$-manifold $N$ is called essential
  if $H_n(c_N)([N])\neq 0\in H_n(B\pi_1(N);\Z)$, where $c_N\colon
  N\longrightarrow B\pi_1(N)$ classifies the universal covering of
  $N$. It is rationally essential if $H_n(c_N)([N])\neq 0$
  in~$H_n(B\pi_1(N);\Q)$.
\end{defn}
Sufficient conditions to ensure essentialness are: asphericity
(obviously), non-zero simplicial volume~\cite{Gromov},
enlargeability~\cite{HKRS,HS}, or the non-vanishing of certain
asymptotic invariants, like the minimal volume entropy or the
spherical volume~\cite{B}. All these properties, except possibly the
last one, actually ensure rational essentialness.

We will show that domination of a manifold by a product implies that
its fundamental group contains many elements with large centralisers.
A simple illustration of this phenomenon is the following result
proved in Section~\ref{proofsec}.
\begin{prop}\label{p:center}
  Let $N$ be a closed, oriented, connected rationally essential
  manifold whose fundamental group has finite centre.  For no closed,
  oriented, connected manifold~$M$ is there a map~$M\times S^1
  \longrightarrow N$ of degree~$1$.
\end{prop}
Results in a similar spirit arise in the study of diffeomorphism
groups~\cite{KKM}.

In order to exclude more general products, we require that the
fundamental group of the target manifold cannot be dominated by a
non-trivial product in the following sense:
  \begin{defn}\label{repgroupsdef}
    An infinite group $\Gamma$ is not presentable by a product if, for
    every homomorphism $\varphi \colon \Gamma_1 \times\Gamma_2
    \longrightarrow\Gamma$ onto a subgroup of finite index, one of the
    factors $\Gamma_i$ has finite image $\varphi
    (\Gamma_i)\subset\Gamma$.
  \end{defn}
  Using Definition~\ref{repgroupsdef}, which will be analysed in
  detail in Section~\ref{s:alg}, we prove the following topological
  result in Section~\ref{proofsec}:
\begin{thm}\label{t:top}
  If $N$ is a closed, oriented, connected rationally essential
  manifold whose fundamental group is not presentable by a product,
  then $P\ngeq N$ for any non-trivial product $P$ of closed, oriented,
  connected manifolds.
\end{thm}
This is complemented by an algebraic result, proved in
Section~\ref{s:alg}, providing examples of groups not presentable by
products:
\begin{thm}\label{t:alg}
  The following groups are not presentable by products:
\begin{enumerate}[\normalfont{(MCG)}]
\item[\normalfont{(H)}] hyperbolic groups that are not virtually
  cyclic,
\item[\normalfont{(N-P)}] fundamental groups of closed Riemannian
  manifolds of non-positive sectional curvature of rank one and of
  dimension~$\geq 2$,
\item[\normalfont{(MCG)}] mapping class groups of closed oriented
  surfaces of genus $\geq 1$.
\end{enumerate}
\end{thm}
The rank occurring in statement (N-P) can be taken to be either the
geometric rank of the Riemannian metric~\cite{BBE}, or the rank of the
fundamental group~\cite{BE}. It is a result of
Ballmann--Eberlein~\cite{BE} that the two agree. In
Section~\ref{s:final}, we discuss the r\^ole of this rank in our
context.

Of course, as the groups in Theorem~\ref{t:alg} are of geometric
origin, the proof uses information gleaned from geometry. The case of
fundamental groups of strictly negatively curved manifolds is
contained as a special case in both (H) and (N-P). For these groups it
is an elementary application of Preissmann's theorem to show that they
are not presentable by products. Our proofs of the cases (H) and (MCG)
follow the same line of argument, using the fact that most elements of
those groups have small centralisers.

Theorem~\ref{t:alg}, particularly statement (H), shows that our
methods are well suited to the study of targets that have some sort of
negative curvature property. This is also true for the applications of
the simplicial volume and of harmonic maps mentioned earlier. However,
in contrast with those techniques, our methods also apply to manifolds
and groups that are not non-positively curved at all. For instance,
the mapping class groups of surfaces of genus~$\geq 2$ occurring in
Theorem~\ref{t:alg} are not hyperbolic because they contain Abelian
subgroups of large rank.  In fact, they do not even have any
semi-simple actions by isometries
on~$\mathrm{CAT}(0)$-spaces~\cite{KL,bh}.

In the second half of this paper we generalise our results following
the philosophy of extending from the hyperbolic to a suitable
semi-hyperbolic situation. In Section~\ref{s:NP} the classical case of
Riemannian manifolds of non-positive sectional curvature is
considered. For such manifolds we prove that being dominated by a
product is equivalent to being virtually diffeomorphic to a product.
The result about closed irreducible locally symmetric spaces of
non-compact type predicted by Gromov is a special case of this more
general theorem.  Similarly, in Section~\ref{s:fibre} we prove that
the total spaces of certain fibrations are dominated by products if
and only if they are virtually trivial.

In Section~\ref{s:sv} we show that our assumptions on the fundamental
group can be relaxed at the expense of making a stronger assumption
than rational essentialness on the manifold we are dealing with.
Finally in Section~\ref{s:final} we discuss the domination relation
more generally, not restricting to product domains, and also
considering targets with finite fundamental groups.

\section{The topological argument}\label{proofsec}

In this section we investigate the relation between presentability by
products on the level of fundamental groups and domination of
rationally essential manifolds by non-trivial products. In particular,
we prove Proposition~\ref{p:center} and Theorem~\ref{t:top}.

Throughout this section, we consider the following situation. We
suppose that $N$ is a closed, oriented, connected $n$-manifold, and
$f\colon X_1\times X_2\longrightarrow N$ is a map of degree~$d\neq 0$
from a non-trivial product of closed, oriented, connected
manifolds~$X_i$.  We choose base points $x_i\in X_i$ and
$f(x_1,x_2)\in N$. All fundamental groups are taken with respect to
these base points. We write
        \begin{align*}
          f_1 & := f\vert_{X_1 \times \{x_2\}} 
                   \colon X_1 \times \{x_2\} \longrightarrow N \ ,\\
          f_2 & := f\vert_{\{x_1\} \times X_2} 
                   \colon \{x_1\} \times X_2 \longrightarrow N \ ,
        \end{align*}
and set $\Gamma_i := \im (\pi_1(f_i))\subset\pi_1(N)$.
Finally, let $\Gamma:=\im (\pi_1(f))\subset\pi_1(N)$.

\begin{lem}\label{l:easy}
  The subgroup $\Gamma$ has finite index in~$\pi_1(N)$ and
  multiplication in~$\pi_1(N)$ defines a surjective homomorphism
  $\varphi\colon\Gamma_1\times\Gamma_2\longrightarrow\Gamma$.
\end{lem}
\begin{proof}
  As $f$ has non-zero degree, covering space theory shows that the
  image~$\Gamma = \im (\pi_1(f))$ has finite index in $\pi_1(N)$.  The
  second statement follows because $\Gamma_1 \cup \Gamma_2$
  generates~$\Gamma$, and the~$\Gamma_i$ commute with each other.
\end{proof}

\begin{prop}\label{p:natural}
  The diagram in Figure~\ref{classfig} is commutative. In particular,
  there are homology classes $\alpha_1 \in \hsing{\dim
    X_1}{B\Gamma_1;\Q}$ and $\alpha_2 \in \hsing{\dim
    X_2}{B\Gamma_2;\Q}$ satisfying
        \[ 
        H_n(c_N)(d \cdot [N]) = H_n(\overline{B\varphi})(\alpha_1 \times\alpha_2 ) \ . 
        \]
\end{prop}
      \begin{figure}
        \begin{center}
          \makebox[0pt]{$%
          \xymatrix@C=1.4em{%
                X_1 \times X_2 
                \ar[rrrrr]^-f
                \ar[dr]^{c_{X_1} \times c_{X_2}}
                \ar[ddd]_-{c_{X_1 \times X_2}}
          &&&&& N 
                \ar[ddd]^-{c_N} 
                \\
              & B\pi_1(X_1) \times B\pi_1(X_2)
                \ar@<.5ex>@{}[rrr]^-{B\pi_1(f_1) \times B\pi_1(f_2)}
                \ar[rrr]
            &&& B\Gamma_1 \times B\Gamma_2 
                \ar@{-->}[ddr]^{\overline{B\varphi}}
                \\
              & B\bigl(\pi_1(X_1) \times \pi_1(X_2)\bigr)
                \ar@{<->}[u]^-{\simeq}
                \ar@<-.5ex>@{}[rrr]_-{B(\pi_1(f_1) \times \pi_1(f_2))}
                \ar[rrr]
                \ar[dl]^-{\ \qquad\smash{B\varphi'}}
            &&& B(\Gamma_1 \times \Gamma_2)
                \ar@{<->}[u]_-{\simeq}
                \ar[dr]_-{\smash{B\varphi}\quad}
                \\
                B\pi_1(X_1 \times X_2) 
                \ar[rrrrr]_-{B\pi_1(f)}
          &&&&& B\pi_1(N)
            }$}
        \end{center}
      \caption{Naturality of classifying maps}
      \label{classfig}
      \end{figure}
      \begin{proof}
        We first explain the notation occurring in
        Figure~\ref{classfig}. For a connected manifold $M$, we
        write~$c_M\colon M \longrightarrow B\pi_1(M)$ for the
        classifying map of the universal covering. Recall that every
        homomorphism~$\psi \colon K' \longrightarrow K$ of groups
        yields a continuous map~$B\psi \colon BK' \longrightarrow BK$
        that induces the given homomorphism~$\psi$ on the level of
        fundamental groups; moreover, $B\psi$ is characterised
        uniquely up to homotopy by this property.

        The vertical homotopy equivalences in the centre of the
        diagram are induced by the projections/inclusions on the level
        of groups.  The map~$B\varphi'$ is induced by the canonical
        isomorphism~$\varphi'\colon\pi_1(X_1) \times \pi_1(X_2)
        \longrightarrow \pi_1(X_1 \times X_2)$ given by the
        inclusions. Finally, $B\varphi$ is the map induced by the
        homomorphism $\varphi$ in Lemma~\ref{l:easy}.
              
        It is a routine matter to verify that the diagram in
        Figure~\ref{classfig} is commutative up to homotopy. By the
        homotopy invariance of homology, the corresponding diagram in
        homology is also commutative.

        In the following, we abbreviate the composition of~$B\varphi$
        with the homotopy equivalence~$B\Gamma_1 \times B\Gamma_2
        \longrightarrow B(\Gamma_1 \times \Gamma_2)$ by
        $\overline{B\varphi} \colon B\Gamma_1 \times B\Gamma_2
        \longrightarrow B\pi_1(N).$ Using the naturality of the
        homological cross-product, we obtain the following relation in
        homology:
        \begin{align*}
            \hsing n {c_N} (d  [N])
        & = \hsing n {c_N} \circ \hsing n {f}([X_1\times X_2])\\
        & = \hsing n {c_N} \circ \hsing n {f}([X_1] \times [X_2])\\
        & =       \hsing n {\overline{B\varphi}}
                  \circ \hsing[big] n {(B\pi_1(f_1) \times B\pi_1(f_2)) \circ (c_{X_1} \times c_{X_2})} 
                  ([X_1]\times [X_2]) \\
        & =       \hsing n {\overline{B\varphi}}(\alpha_1 \times \alpha_2),
        \end{align*}
        where we put $\alpha_i := \hsing {\dim X_i} {B\pi_1(f_i) \circ
          c_{X_i}}([X_i])\in \hsing{\dim X_i}{B\Gamma_i;\Q}$.
      \end{proof}
      
The proofs of the topological results stated in the introduction
are now completely straightforward.

\begin{proof}[Proof of Proposition~\ref{p:center}]
  If $f\colon M\times S^1\longrightarrow N$ has degree one, then
  $B\pi_1(f)$ is surjective, and so it sends the fundamental group of
  the circle to the centre~$C(\pi_1(N))$ of~$\pi_1(N)$. As this centre
  is assumed to be finite, its classifying space has trivial rational
  homology and we conclude
$$
\hsing[big]{1} {B\pi_1(f\vert_{S^1}) \circ c_{S^1}}([S^1])= 0 \in  \hsing[big]{1}{BC(\pi_1(N));\Q} \ .
$$ 
Now Proposition~\ref{p:natural} (with $d=1$) shows that $\hsing n
{c_N} ([N])$ vanishes, contradicting the rational essentialness of
$N$.
\end{proof}

\begin{proof}[Proof of Theorem~\ref{t:top}]
  Let $N$ be a closed, oriented, connected manifold, and assume that
  there exists a map~$f\colon X_1 \times X_2 \longrightarrow N$ of
  non-zero degree. If $N$ is rationally essential, then the homology
  classes $\alpha_i$ in Proposition~\ref{p:natural} are both
  non-trivial and in positive degrees. Therefore the groups $\Gamma_i$
  are both infinite. Now Lemma~\ref{l:easy} shows that $\pi_1(N)$ is
  presented by the product $\Gamma_1\times\Gamma_2$.
\end{proof}

\section{Groups not presentable by products}\label{s:alg}

In this section we discuss groups not presentable by products.  In
particular we prove Theorem~\ref{t:alg}, and we generalise the result
to certain groups arising as extensions.

\subsection*{Preliminaries}

Consider a homomorphism
$\varphi\colon\Gamma_1\times\Gamma_2\longrightarrow\Gamma$.  Without
loss of generality we can replace each $\Gamma_i$ by its image in
$\Gamma$ under the restriction of~$\varphi$, so that we may assume the
factors $\Gamma_i$ to be subgroups of $\Gamma$ and $\varphi$ to be
multiplication in $\Gamma$.

\begin{lem}\label{l:C}
  If $\Gamma$ is not presentable by a product, then every finite index
  subgroup has finite centre.
\end{lem}
\begin{proof}
  If $\bar\Gamma \subset \Gamma$ is a subgroup of finite index with
  infinite centre~$C(\bar\Gamma)$, then the multiplication
  map~$\bar\Gamma\times C(\bar\Gamma)\longrightarrow\Gamma$ shows that
  $\Gamma$ is presentable by a product.
\end{proof}

The following is a kind of converse to this observation:
\begin{prop}\label{p:C}
  If every subgroup of finite index in $\Gamma$ has trivial centre,
  then $\Gamma$ is irreducible if and only if it is not presentable by
  a product.
\end{prop}
Before the proof of this proposition, we need another lemma.
\begin{lem}\label{l:center}
  Let $\Gamma_1$, $\Gamma_2\subset\Gamma$ be commuting subgroups with
  the property that $\Gamma_1 \cup \Gamma_2$ generates~$\Gamma$. Then
  the multiplication
  homomorphism~$\varphi\colon\Gamma_1\times\Gamma_2\longrightarrow\Gamma$
  is well-defined and surjective and the following statements hold:
\begin{enumerate}
\item the intersection~$\Gamma_1\cap \Gamma_2\subset\Gamma$ is a
  subgroup of the centre of~$\Gamma$, and
\item the kernel of~$\varphi$ is isomorphic to the Abelian
  group~$\Gamma_1\cap \Gamma_2$.
\end{enumerate}
\end{lem}
\begin{proof}
  The first statement is true because the~$\Gamma_i$ generate~$\Gamma$
  and commute with each other.  So, if an element of $\Gamma$ is in
  both $\Gamma_i$, then it commutes with all generators. It follows in
  particular that $\Gamma_1\cap\Gamma_2$ is Abelian.

  For the second statement observe that
  $(g_1,g_2)\in\Gamma_1\times\Gamma_2$ maps to the neutral element
  of~$\Gamma$ if and only if $g_1=g_2^{-1}$ in~$\Gamma$. But
  $g_1\in\Gamma_1$, $g_2^{-1}\in\Gamma_2$, which
  implies~$g_1$,~$g_2\in\Gamma_1\cap\Gamma_2$. Thus, the
  anti-diagonal~$\Gamma_1 \cap \Gamma_2 \longrightarrow \Gamma$, $g
  \longmapsto (g, g^{-1})$ is an isomorphism onto the kernel
  of~$\varphi$.
\end{proof}

\begin{proof}[Proof of Proposition~\ref{p:C}]
  Suppose that such a $\Gamma$ is reducible in the sense that it has a
  finite index subgroup that is a direct product of infinite groups.
  Then obviously $\Gamma$ is presentable by a product.

  Conversely, assume that
  $\varphi\colon\Gamma_1\times\Gamma_2\longrightarrow\Gamma$ is
  surjective onto a subgroup~$\bar\Gamma\subset\Gamma$ of finite
  index, and that $\varphi(\Gamma_i)$ is infinite for both $i$. Then
  Lemma~\ref{l:center} applied to the subgroups $\varphi(\Gamma_1)$
  and $\varphi(\Gamma_2)$ of~$\bar\Gamma$ shows that $\Gamma$ must be
  reducible, because we assumed that all finite index subgroups have
  trivial centre.
\end{proof}

Sometimes it is convenient to replace a given group by a subgroup of
finite index.  This transition does not affect presentability by
products by the following straightforward observation:
\begin{lem}\label{l:finindexpres}
  Let $\Gamma$ be a group. A finite index subgroup $\bar\Gamma\subset\Gamma$ is
  presentable by a product if and only if $\Gamma$ is.
\end{lem}

\subsection*{Hyperbolic groups}

We now show that most hyperbolic groups are not presentable by
products, which is case (H) of Theorem~\ref{t:alg}. The following
lemma is probably well known, but we did not find it explicitly in the
literature.

  \begin{lem}\label{hypcentrelem}
    Let $\Gamma$ be a hyperbolic group that is not virtually cyclic.
    Then the centre $C(\Gamma)$ is finite.
  \end{lem}
  \begin{proof}
    Suppose the centre~$C(\Gamma)$ contains a non-trivial element~$g$.
    By definition, the centraliser~$C_{\Gamma}(g)$ of such a $g$ is
    the whole group~$\Gamma$. If $g$ had infinite order, then
    $C_{\Gamma}(g)=\Gamma$ would be virtually
    cyclic~\cite[Corollary~3.10 on~p.~462]{bh}, contradicting the
    hypothesis on~$\Gamma$. Hence, all elements of~$C(\Gamma)$ have
    finite order.

    As $\Gamma$ is not virtually cyclic, $\Gamma$ must be infinite;
    thus, $\Gamma$ contains an element~$g$ of infinite
    order~\cite[Proposition~2.22 on~p.~458]{bh}. The torsion
    group~$C(\Gamma)$ is a subgroup of the
    centraliser~$C_{\Gamma}(g)$, which is virtually cyclic. So
    $C(\Gamma)$ is finite.
  \end{proof}

\begin{prop}\label{hypgroupprop}
  A hyperbolic group that is not virtually cyclic is not presentable
  by a product.
\end{prop}
\begin{proof}
  Let $\Gamma$ be a hyperbolic group that is not virtually cyclic and
  suppose that $\varphi \colon\Gamma_1 \times\Gamma_2
  \longrightarrow\Gamma$ is a homomorphism onto a finite index
  subgroup.  Because finite index subgroups of hyperbolic groups are
  hyperbolic, we may assume that $\varphi$ is surjective.
  Furthermore, we may assume without loss of generality that
  the~$\Gamma_i$ are subgroups of~$\Gamma$ and that $\varphi$ is just
  the multiplication map.
 
  By Lemma~\ref{l:center} the intersection $\Gamma_1\cap\Gamma_2$ is
  contained in the centre of~$\Gamma$, and is therefore finite by
  Lemma~\ref{hypcentrelem}.

  Because $\Gamma$ is not finite, it contains an element of infinite
  order~\cite[Proposition~2.22 on~p.~458]{bh}. Therefore we may assume
  that $\Gamma_1$ contains an element~$g_1$ of infinite order. The
  group~$\Gamma_2$ is contained in~$C_{\Gamma}(g_1)$ and hence is
  virtually cyclic~\cite[Corollary~3.10 on~p.~462]{bh}.

  Now assume that $\Gamma_2$ is also infinite. Then $\Gamma_2$
  contains an element~$g_2$ of infinite order (because the
  group~$\Gamma_2$ is virtually cyclic). Now $g_1$ and $g_2$ generate
  a copy of~$\bZ \times \bZ$ in~$\Gamma$. For if this were false, then
  a power of~$g_1$ would equal a power of~$g_2$, and these powers
  would be contained in the finite group~$\Gamma_1\cap \Gamma_2$; this
  is not possible because the~$g_i$ have infinite order.  However,
  hyperbolic groups cannot contain~$\bZ \times \bZ$ because
  centralisers of elements of infinite order are virtually cyclic.
  This contradiction shows that one of the $\Gamma_i$ must be finite.
\end{proof}

\subsection*{Fundamental groups of non-positively curved manifolds}
      
The following proposition corresponds to the case~(N-P) of
Theorem~\ref{t:alg}.
\begin{prop}\label{p:NP}
  If $\Gamma$ is the fundamental group of a closed Riemannian
  manifold~$N$ of non-positive sectional curvature of rank one and of
  dimension $\geq 2$, then $\Gamma$ is not presentable by a product.
  \end{prop}
\begin{proof}
  Recall that by the result of Ballmann--Eberlein~\cite{BE} the
  geometric rank of~$N$ coincides with the algebraic rank of~$\Gamma$,
  that this rank is additive under direct products of manifolds
  respectively of groups, and that it is invariant under passage to
  finite coverings respectively to finite index subgroups. The
  assumption that $N$ be of rank one therefore implies that $N$ is
  locally irreducible and that $\Gamma$ is irreducible (any direct
  factor of~$\Gamma$ would be infinite because the group is
  torsion-free).

  The irreducibility of~$N$ and the assumption~$\dim N\geq 2$ imply
  that $N$ has no Euclidean local de Rham factor, so that by the
  result of Eberlein~\cite[p.~210f]{EbJDG}, the centres of~$\Gamma$
  and of all its finite index subgroups are trivial. Now the
  conclusion follows from Proposition~\ref{p:C}.
\end{proof}

We will generalise this Proposition in the course of the proof of
Theorem~\ref{t:N-P} in Section~\ref{s:NP}.

\subsection*{Mapping class groups}

Finally we deal with the case (MCG) in Theorem~\ref{t:alg}. We refer
the reader to Ivanov's book~\cite{MCG} for the relevant information on
mapping class groups and their subgroups.

A mapping class is called irreducible if it does not fix any
non-trivial isotopy class of a curve on the surface.  The same
terminology is applied to subgroups of the mapping class group.  Thus
a subgroup is said to be irreducible if there is no isotopy class of a
curve on the surface fixed by the whole subgroup.
\begin{prop}\label{p:MCG}
  Mapping class groups of closed, oriented surfaces of genus at
  least~$1$ are not presentable by products.
\end{prop}
\begin{proof}
  In genus one the mapping class group is $\mathrm{SL}_2(\Z)$. This
  group is virtually free and hence hyperbolic.  Thus we may appeal to
  Proposition~\ref{hypgroupprop} to cover this case.

  Now consider the mapping class group of a closed surface of
  genus~$\geq 2$.  After passing to a subgroup of finite index, we may
  assume that we are dealing with a torsion-free group $\Gamma$
  (compare Lemma~\ref{l:finindexpres}).  Assume that $\Gamma_1$,
  $\Gamma_2\subset\Gamma$ are non-trivial commuting subgroups for
  which the
  multiplication~$\Gamma_1\times\Gamma_2\longrightarrow\Gamma$ is
  surjective.

  As each $\Gamma_i$ is infinite and normal in the irreducible
  group~$\Gamma$, it is itself irreducible~\cite[Corollary~7.13]{MCG}.
  It follows that each~$\Gamma_i$ contains a pseudo-Anosov
  element~\cite[Corollary~7.14]{MCG}.  Now if $g_1\in\Gamma_1$ is
  pseudo-Anosov, then the centraliser $C_{\Gamma}(g_1)$ of $g_1$ in
  $\Gamma$ is cyclic~\cite[Lemma~8.13]{MCG} and contains~$\Gamma_2$.
  Thus $\Gamma_2$ is cyclic. Reversing the r\^oles of~$\Gamma_1$ and
  $\Gamma_2$ we see that $\Gamma_1$ is cyclic as well. Thus we reach
  the absurd conclusion that the mapping class group is virtually
  Abelian.  This contradiction shows that the mapping class group is
  not presentable by a product.
\end{proof}

\subsection*{Group extensions}

Having completed the proof of Theorem~\ref{t:alg}, we now extend our
results to other groups and manifolds.  In this direction, we will use
the following result to study fibrations in Section~\ref{s:fibre}:
\begin{prop}\label{p:fibre}
  Assume that the group~$\Gamma$ fits into an extension of the form
\begin{equation}\label{ext}
  1\longrightarrow K\longrightarrow \Gamma\stackrel{\pi}{\longrightarrow} Q\longrightarrow 1 \ ,
\end{equation}
where both $K$ and $Q$ are torsion-free, non-trivial and not
presentable by products. Then $\Gamma$ is presentable by a product if
and only if it has a finite index subgroup that is a direct product of
finite index subgroups of $K$ and $Q$.
\end{prop}
\begin{proof}
  One direction is clear: if $\Gamma$ has a finite index subgroup that
  is a direct product of non-trivial torsion-free groups, then it is
  presented by a product.

  For the converse, let
  $\varphi\colon\Gamma_1\times\Gamma_2\longrightarrow\Gamma$ be
  surjective onto a finite index subgroup.  After replacing $\Gamma$
  by this subgroup (and suitably replacing $K$ and $Q$ by the
  corresponding finite index subgroups), we may assume that $\varphi$
  is surjective. As before, we may also assume that the $\Gamma_i$ are
  subgroups of $\Gamma$.  Assume that they are both non-trivial.

  Now consider
  $\pi\circ\varphi\colon\Gamma_1\times\Gamma_2\longrightarrow Q$. This
  is surjective.  As $Q$ is torsion-free and not presentable by a
  product, one of the $\Gamma_i$ has trivial image.  Let us assume
  that this is $\Gamma_1$. Then by the exactness of~\eqref{ext}, we
  have $\Gamma_1\subset K$.

  If $\Gamma_1$ and $K\cap\Gamma_2$ are both infinite, then we have a
  contradiction with the assumption that $K$ is not presentable by a
  product. So one of these groups is finite. As $K$ is torsion-free,
  such a finite subgroup must be trivial. But $\Gamma_1$ is
  non-trivial by assumption, so we conclude that $K\cap\Gamma_2$ is
  trivial, and so $\pi$ maps $\Gamma_2$ injectively onto $Q$.

  We now claim that
  $\varphi\colon\Gamma_1\times\Gamma_2\longrightarrow\Gamma$ is an
  isomorphism.  For if $\varphi (g_1,g_2)$ is trivial, then so
  is~$\pi(\varphi (g_1,g_2)) = (\pi\vert\Gamma_2)(g_2)$, which implies
  that $g_2$ is trivial. Hence, the triviality of~$\varphi
  (g_1,g_2)=\varphi (g_1,e)$ shows that $g_1$ is trivial
  in~$\Gamma_1$.
\end{proof}

As an immediate application, we can extend Proposition~\ref{p:MCG} to
mapping class groups of surfaces with a marked point. Here we consider
diffeomorphisms of a surface fixing the marked point, up to isotopies
fixing the point.
\begin{cor}\label{c:mp}
  The mapping class group of a closed, oriented surface of genus $\geq
  2$ with a marked point is not presentable by a product.
\end{cor}
\begin{proof}
  Let $\Gamma$ be the mapping class group of a closed, oriented
  surface~$\Sigma$ with respect to a marked point.  We have an exact
  sequence
\begin{equation}\label{eq:mp}
  1\longrightarrow\pi_1(\Sigma)\longrightarrow \Gamma\stackrel{\pi}{\longrightarrow} Q\longrightarrow 1 \ ,
\end{equation}
where $Q$ is the mapping class group of $\Sigma$ (without a marked
point) considered in Proposition~\ref{p:MCG}, and $\pi$ is the
forgetful map.  We may pull back this extension to a torsion-free
finite index subgroup of $Q$, so that the assumption on the quotient
in Proposition~\ref{p:fibre} is satisfied. Note that the kernel is the
fundamental group of a closed manifold of negative curvature, and so
is not presentable by a product, for example by
Proposition~\ref{hypgroupprop}.  The conclusion follows from
Proposition~\ref{p:fibre}, because the (universal)
extension~\eqref{eq:mp} is non-trivial when restricted to any finite
index subgroups of $\pi_1(\Sigma)$ and of
$Q=\mathrm{Out}(\pi_1(\Sigma))$.
\end{proof}

\section{Manifolds of non-positive curvature}\label{s:NP}

In this section we discuss the domination of closed manifolds of
non-positive sectional curvature by products.  The combination of
Theorems~\ref{t:top} and \ref{t:alg} shows that $P\ngeq N$ whenever
$P$ is a non-trivial product of closed manifolds and $N$ is a
non-positively curved closed Riemannian manifold of rank one; this
includes in particular all closed Riemannian manifolds of negative
sectional curvature. Most non-positively curved manifolds are of rank
one, even when they have a great deal of zero
curvature~\cite{BBE,BE,FS}, but it remains to deal with the general
case. The following result shows that domination by a product is only
possible for a non-positively curved manifold if it is itself
virtually a product.

\begin{thm}\label{t:N-P}
  Let $N$ be a closed, oriented, connected Riemannian manifold with
  non-positive sectional curvature and $\Gamma$ its fundamental group.
  If the dimension of $N$ is at least two, then the following
  properties are equivalent:
\begin{enumerate}
\item $P\geq N$ for some non-trivial product~$P$ of closed, oriented
  manifolds,
\item the fundamental group $\Gamma$ is presentable by a product,
\item some finite index subgroup of~$\Gamma$ splits as a non-trivial direct product,
\item there is a finite covering of~$N$ diffeomorphic to a non-trivial
  product~$N_1\times N_2$ of closed, oriented manifolds~$N_i$.
\end{enumerate}
\end{thm}
\begin{proof}
  Recall that by the Cartan--Hadamard theorem manifolds of
  non-positive curvature are aspherical, and so are classifying spaces
  for their fundamental groups. Therefore, (1) implies (2) by
  Theorem~\ref{t:top}.

  Assume that (2) holds. Then there is a
  homomorphism~$\varphi\colon\Gamma_1\times\Gamma_2\longrightarrow\Gamma$
  that is surjective onto a finite index
  subgroup~$\bar\Gamma\subset\Gamma$ such that both $\Gamma_i$ have
  infinite image.  As discussed in Section~\ref{s:alg} we may take the
  $\Gamma_i$ to be subgroups of $\Gamma$ and $\varphi$ to be
  multiplication. If $\bar\Gamma$ has trivial centre, then
  Lemma~\ref{l:center} shows that $\varphi$ is injective, and so
  $\bar\Gamma$ is a non-trivial direct product. If $\bar\Gamma$ has
  non-trivial centre, then the centre is infinite because $\Gamma$ is
  torsion-free.  A result of Eberlein~\cite{E} shows that $\bar\Gamma$
  has a finite index subgroup that splits off $C(\bar\Gamma)$ as a
  direct factor. This gives a non-trivial splitting of the subgroup as
  a direct product, because we assumed $\dim N>1$, which implies that
  $\Gamma$ is not virtually~$\Z$. Thus (2) implies (3).

  Next, assume that (3) holds, so that there is a finite index
  subgroup~$\bar \Gamma$ of~$\Gamma$ that splits as a non-trivial
  direct product $\Gamma_1\times\Gamma_2$. If the group $\bar\Gamma$
  has trivial centre, then the Gromoll--Wolf~\cite{GW} splitting
  theorem gives an isometric splitting $\bar N = N_1\times N_2$, where
  $\bar N$ is the covering of $N$ corresponding to the
  subgroup~$\bar\Gamma$, and $\pi_1(N_i)=\Gamma_i$. If the centre of
  $\bar\Gamma$ is non-trivial, then again by Eberlein's
  results~\cite{E} some finite covering splits off a torus as a direct
  factor. (This splitting is usually not isometric~\cite{LY,E}.) Thus
  (3) implies (4).

  Finally, the implication from (4) to (1) is trivial.
\end{proof}
The implication from (1) to (4) shows that manifolds of non-positive
curvature dominated by products are virtually diffeomorphic to
products. As an immediate consequence of this, we confirm Gromov's
suggestion~\cite[5.36 on~p.~303f]{gromovmetric}:
\begin{cor}\label{c:locsym}
  Let $N$ be a closed locally symmetric space of non-compact type.
  Then $P\geq N$ for some non-trivial product $P$ if and only if $N$
  is reducible in the sense that it has a finite covering isometric to
  a product.
\end{cor}
\begin{proof}
  Theorem~\ref{t:N-P} shows that $P\geq N$ is equivalent to $N$ having
  a finite covering diffeomorphic to a non-trivial product. In the
  case that $\pi_1(N)$ has trivial centre, the proof of
  Theorem~\ref{t:N-P} also shows that this diffeomorphic splitting is
  in fact isometric.  Thus we only have to check that the centre
  of~$\pi_1(N)$ is trivial.

  For $N$ to be a locally symmetric space of non-compact type means
  that the universal covering $\tilde N$ is a globally symmetric space
  without compact or Euclidean factors in its de Rham decomposition.
  There are two ways to see that this implies the triviality of the
  centre of $\pi_1(N)$, geometrically or algebraically.
  Geometrically, the centre of the fundamental group of any closed
  manifold of non-positive curvature is contained in the Clifford
  subgroup, which, by a result of Eberlein~\cite{EbJDG}, is of rank
  equal to the dimension of the Euclidean de Rham factor of $\tilde
  N$. Algebraically, for a locally symmetric space the fundamental
  group is a lattice in a semisimple Lie group with finite centre. In
  the case of a symmetric space with no compact or Euclidean factors
  the isometry group does not have any compact factors either, and
  standard results about lattices show that torsion-free cocompact
  lattices have trivial centre~\cite[Corollary~5.18]{raghunathan},
  \cite[Corollary~4.41]{wittemorris}.
\end{proof}

Combining our results with a result of Ballmann--Eberlein~\cite{BE},
we have the following:
\begin{cor}
  Let $N$ be a closed, oriented, connected Riemannian manifold of
  non-positive sectional curvature, and $\Gamma$ its fundamental
  group.  The four conditions in Theorem~\ref{t:N-P} hold for $N$,
  respectively for $\Gamma$, if and only if $N$, respectively
  $\Gamma$, is of rank $\geq 2$ and $N$ is not an irreducible locally
  symmetric space of non-compact type.
\end{cor}
\begin{proof}
  If $N$, respectively $\Gamma$, is of rank one, then condition (2) in
  Theorem~\ref{t:N-P} does not hold by Proposition~\ref{p:NP}.  If $N$
  is an irreducible locally symmetric space of non-compact type, then,
  by the definition of irreducibility, condition~(3) does not hold.
  In all other cases, a result of
  Ballmann--Eberlein~\cite[Theorem~C]{BE} shows that condition~(3)
  does hold.
\end{proof}

Theorem~\ref{t:N-P} has the following extension to rationally
essential manifolds:
\begin{cor}\label{c:nonpos}
  Let $M$ be a closed, oriented, manifold whose fundamental group is
  also the fundamental group of a closed manifold $N$ admitting a
  Riemannian metric of non-positive sectional curvature.  If $M$ is
  rationally essential and $P\geq M$ for a non-trivial product~$P$ of
  closed manifolds, then $N$ has a finite covering $\bar N$ that is
  diffeomorphic to a non-trivial product.
\end{cor}
\begin{proof}
  Theorem~\ref{t:top} shows that $P\geq M$ implies statement (2) from
  Theorem~\ref{t:N-P} for the fundamental group $\pi_1(M)=\pi_1(N)$.
  The implication from (2) to (4) in Theorem~\ref{t:N-P} gives the
  conclusion.
\end{proof}
\begin{rem}
  In the case that every finite index subgroup of $\pi_1(N)$ has
  trivial centre, the combination of Theorems~\ref{t:top}
  and~\ref{t:N-P} shows that in the situation of
  Corollary~\ref{c:nonpos} with $P=X_1\times X_2\geq M$, each $X_i$ is
  rationally essential in the corresponding factor of a finite
  covering of $N$.
\end{rem}

To end this section, we point out that many non-positively curved
Riemannian manifolds of rank one are comparable to direct products in
the relation~$\geq$. For example, closed hyperbolic Riemann
surfaces~$N$ satisfy~$N\geq S^1\times S^1$.  More interestingly, the
branched covering construction of Fornari--Schroeder~\cite{FS}
produces rank one manifolds~$N$ of higher dimension that are close to
products in a certain geometric sense, and that by construction
satisfy~$N\geq (X_1\times X_2)$ for suitable $X_i$.

\section{Fibre bundles}\label{s:fibre}

In this section we consider fibrations whose base and fibre have
fundamental groups not presentable by products.

\begin{thm}\label{t:fibre}
  Let $M$ be a closed, oriented, connected manifold that is the total
  space of a fibration whose base~$B$ and fibre~$F$ are closed,
  oriented, connected aspherical manifolds. Assume that $\pi_1(B)$ and
  $\pi_1(F)$ are not presentable by products. If $P\geq M$ for some
  non-trivial product~$P=X_1\times X_2$ of closed, oriented, connected
  manifolds, then $M$ is finitely covered by a product whose factors
  are finite coverings of~$B$ and $F$ respectively and are dominated
  by the~$X_i$.
\end{thm}
\begin{rem}
  In the same way that Theorem~\ref{t:N-P} implies
  Corollary~\ref{c:nonpos}, Theorem~\ref{t:fibre} has consequences for
  rationally essential manifolds in place of aspherical ones.
\end{rem}
\begin{proof}
  Suppose there are closed, oriented, connected manifolds~$X_i$ such
  that their product~$P=X_1\times X_2$ admits a map~$f\colon
  P\longrightarrow M$ of degree $d\neq 0$.
  Proposition~\ref{p:natural} shows that the images $\Gamma_i$ of the
  fundamental groups of the factors $X_i$ are both infinite and
  together generate a finite index subgroup~$\Gamma$ of~$\pi_1(M)$.
  Applying Proposition~\ref{p:fibre} (and its proof) to~$\Gamma$ we
  see that $\Gamma \cong \Gamma_1 \times \Gamma_2$ and that $\Gamma_1$
  and $\Gamma_2$ are finite index subgroups of~$\pi_1(B)$ and
  $\pi_1(F)$ respectively. In particular, $M$ is finitely covered by a
  manifold~$\bar M$ whose fundamental group is the direct product of
  the finite index subgroups~$\Gamma_i$ of~$\pi_1(B)$ and $\pi_1(F)$
  respectively.

  As the base $B$ and fibre $F$ are aspherical, so is $M$. It follows
  that $\bar M$ is homotopy equivalent to the product of finite covers
  of~$B$ and~$F$ corresponding to~$\Gamma_1$ and~$\Gamma_2$.
  Proposition~\ref{p:natural} shows that these factors are dominated
  by the~$X_i$.
\end{proof}
The above proof shows that a finite covering of $M$ is homotopy
equivalent to a trivial bundle. In specific situations one can
sometimes prove more. In this direction we have for example:
\begin{cor}\label{c:surfbundle}
  Let $M$ be a surface bundle over a surface with base $B$ and fibre
  $F$ both of genus $\geq 2$, and let $\Gamma:=\pi_1(M)$.  The
  following are equivalent:
\begin{enumerate}
\item $P\geq M$ for some non-trivial product~$P$ of closed, oriented
  manifolds,
\item the fundamental group $\Gamma$ is presentable by a product,
\item a finite index subgroup of~$\Gamma$ splits as a non-trivial
  direct product,
\item there is a finite covering of $M$ diffeomorphic to a trivial
  surface bundle.
\end{enumerate}
\end{cor}
It is well known that the last condition is equivalent to the
finiteness of the image of the monodromy representation of $\pi_1(B)$
in the mapping class group of $F$.
\begin{proof}
  As the base and fibre are negatively curved, their fundamental
  groups are not presentable by products by
  Proposition~\ref{hypgroupprop}. The proof of Theorem~\ref{t:fibre}
  shows that the first three conditions are equivalent, and that they
  are equivalent to $M$ having a finite cover that is homotopy
  equivalent to trivial surface bundle. By a result of
  Hillman~\cite[Corollary~5.6.4 on p.~94]{H} we may assume that this
  homotopy equivalence is a fibre-preserving diffeomorphism. (The
  assumption $\chi(B) < \chi(F)$ in Hillman's result can always be
  arranged by pulling back to a finite covering of the base.)
\end{proof}
Corollary~\ref{c:surfbundle} is interesting from several different
points of view. First of all, there are many examples of surface
bundles over surfaces with both fibre and base of genus~$\geq 2$ that
do not admit any metric of non-positive sectional curvature~\cite{KL}.
Thus Corollary~\ref{c:surfbundle} cannot be deduced from the results
of the previous section, and Theorem~\ref{t:fibre} is of a different
nature than the differential-geometric results of that section. Second
of all, no surface bundle over a surface is known to admit a
negatively curved metric, so that Corollary~\ref{c:surfbundle} has no
known overlap with results about maps of products to negatively curved
targets.

In higher dimensions, Theorem~\ref{t:fibre} and its potential
generalisations from bundles to other aspherical manifolds raise the
following question: If $M$ is a closed, oriented, connected,
aspherical manifold whose fundamental group splits as a non-trivial
direct product~$\Gamma_1 \times \Gamma_2$, can $M$ be split up to
homotopy or homeomorphism as a product of closed, oriented, connected
manifolds with fundamental group~$\Gamma_1$ and $\Gamma_2$
respectively? Using the sophisticated machinery developed in the field
of topological rigidity, this question can be answered affirmatively
for a large class of such manifolds~\cite{lueck}.  This solution
relies on deep results concerning the Farrell-Jones conjecture, the
Borel conjecture, the Novikov conjecture, and the resolution of
homology manifolds. For non-positively curved manifolds, we have seen
in the previous section that the geometric arguments in the proof of
the Gromoll--Wolf splitting theorem~\cite{GW} take care of these
complications.

\section{Amenable centralisers}\label{s:sv}

In our discussion so far we have typically assumed that the target
manifolds are rationally essential, and that their fundamental groups
have small centralisers. For the groups occurring in
Theorem~\ref{t:alg}, the centralisers are as small as possible, that
is (virtually) cyclic. This bound on the size of centralisers was not
needed for all elements, but for the generic ones, e.g., the elements
of infinite order in hyperbolic groups, or the pseudo-Anosov elements
in mapping class groups.

In this section we treat groups with larger centralisers, which are
not necessarily virtually Abelian.  The price to be paid is that
rational essentialness has to be replaced by a stronger assumption.
As an example of such a generalisation we consider fundamental groups
with amenable centralisers and manifolds with non-zero simplicial
volume.  For background on the simplicial volume we refer to the work
of Gromov~\cite{Gromov} and the exposition by Ivanov~\cite{ivanov},
and for information on amenability to Paterson's book~\cite{paterson}.
\begin{thm}\label{svthm}
  Let $M$ be a closed, oriented, connected manifold with non-zero
  simplicial volume.  Assume that $\pi_1(M)$ contains an element of
  infinite order, and that every element of infinite order has
  amenable centraliser. Then $P\ngeq M$ for any non-trivial
  product~$P$ of closed, oriented, connected manifolds.
\end{thm}
\begin{proof}
  Suppose for a contradiction that $P=X_1\times X_2$ is a non-trivial
  product of closed, oriented, connected manifolds admitting a
  map~$f\colon P\longrightarrow M$ of degree~$d\neq 0$. By
  Lemma~\ref{l:easy} we have commuting
  subgroups~$\Gamma_1$,~$\Gamma_2\subset\pi_1(M)$ such that the
  multiplication map~$\Gamma_1\times\Gamma_2\longrightarrow\pi_1(M)$
  is surjective onto a subgroup~$\Gamma\subset\pi_1(M)$ of finite
  index.

  The assumption that $\pi_1(M)$ contains an element of infinite order
  implies that the same is true for one of the~$\Gamma_i$. If
  $g_1\in\Gamma_1$ has infinite order, then by the assumption about
  centralisers, the centraliser~$C_{\Gamma}(g_1)$ of~$g_1$
  in~$\Gamma$, which is a subgroup of the amenable
  group~$C_{\pi_1(M)}(g_1)$, is amenable.  But $\Gamma_2\subset
  C_{\Gamma}(g_1)$, because the $\Gamma_i$ commute. Therefore
  $\Gamma_2$ is amenable.

  Proposition~\ref{p:natural} also holds for homology with real
  coefficients, and so we have homology classes $\alpha_1 \in
  \hsing{\dim X_1}{B\Gamma_1;\R}$ and $\alpha_2 \in \hsing{\dim
    X_2}{B\Gamma_2;\R}$ satisfying
\[ 
  H_n(c_M)(d \cdot [M]) = H_n(\overline{B\varphi})(\alpha_1 \times\alpha_2 ) \ . 
\]
We now apply the functorial $\ell^1$-semi-norm to this equation. On
the one hand, by the assumption on the simplicial volume of~$M$ and
the degree~$d$, the mapping theorem for bounded
cohomology~\cite[p.~40/17]{Gromov} yields
\[
  \lone[big]{\hsing n {c_M}(d \cdot [M])} = \vert d \vert \cdot   \sv M \neq 0 \ .
\]
On the other hand, $ \lone{\alpha_2}=0$ because $\Gamma_2$ is
amenable~\cite[p.~40/17]{Gromov}, \cite[Theorem~4.3]{ivanov}, which
implies
\[
  \lone[big]{\hsing n {\overline{B\varphi}}(\alpha_1 \times\alpha_2)} 
  \leq 2^{\dim M} \cdot \lone{\alpha_1} \cdot \lone{\alpha_2} = 0 
\]
by the almost-multiplicativity of the norm.  This contradiction
completes the proof that $P\ngeq M$.
\end{proof}

The assumption that all elements of infinite order have amenable
centralisers is quite restrictive, but can in practice be relaxed. For
example, we have the following variation on Theorem~\ref{svthm}:
\begin{cor}\label{c:amen}
  Let $M$ be a closed, oriented, connected manifold with non-zero
  simplicial volume.  Assume that $\pi_1(M)$ fits into an extension
\begin{equation}\label{eq:ext}
1\longrightarrow K\longrightarrow\pi_1(M)\stackrel{\pi}{\longrightarrow} Q\longrightarrow 1 \ ,
\end{equation}
where $K$ is amenable and $Q$ is a torsion-free group in which every
non-trivial element has an amenable centraliser.  Then $P\ngeq M$ for
any non-trivial product~$P$ of closed, oriented, connected manifolds.
\end{cor}
\begin{proof}
  If $f\colon P\longrightarrow M$ is of non-zero degree, we can again
  apply Lemma~\ref{l:easy} to find commuting
  subgroups~$\Gamma_1$,~$\Gamma_2\subset\pi_1(M)$ for which the
  multiplication map~$\Gamma_1\times\Gamma_2\longrightarrow\pi_1(M)$
  is surjective onto a subgroup of~$\pi_1(M)$ of finite index. It
  follows that $\pi(\Gamma_1)$ and $\pi(\Gamma_2)$ are commuting
  subgroups of $Q$ generating a finite index subgroup. If
  $\pi(\Gamma_1)$ is non-trivial, then it follows from the assumption
  about centralisers in $Q$ that $\pi(\Gamma_2)$ is amenable. As the
  class of amenable groups is closed under extensions, we conclude
  that $\Gamma_2$ is amenable, which gives a contradiction as in the
  proof of Theorem~\ref{svthm}. If $\pi(\Gamma_1)$ is trivial, then
  $\Gamma_1\subset K$ is amenable because it is a subgroup of an
  amenable group, and we again have a contradiction.
\end{proof}
Corollary~\ref{c:amen} excludes domination by products for many
manifolds whose fundamental groups are presentable by products. For
example, the extension~\eqref{eq:ext} could be central, so that $K$ is
Abelian and in the centre of~$\pi_1(M)$. In this situation
Lemma~\ref{l:C} implies that $\pi_1(M)$ is presentable by a product,
so that Theorem~\ref{t:top} cannot be used to prove~$P\ngeq M$. In
general, subexponential extensions in the sense of
Sambusetti~\cite{Sam} provide examples of amenable extensions, because
any group of subexponential growth is amenable.

Here is a concrete example demonstrating the applicability of
Corollary~\ref{c:amen}.
\begin{ex}
  Let $X=\C H^2/Q$ be a smooth compact complex ball quotient. Then $X$
  is a complex-algebraic surface whose fundamental group $Q$ is
  torsion-free and hyperbolic (and therefore has virtually cyclic
  centralisers). In particular it satisfies the assumptions made on
  $Q$ in Corollary~\ref{c:amen}. Let $Y$ be an elliptic curve, and
  $M\subset X\times Y$ a smooth hyperplane section. Then
  $\pi_1(M)=\pi_1(X\times Y)=Q\times\Z^2$ by the Lefschetz hyperplane
  theorem.  Moreover, the projection of $X\times Y$ to the first
  factor restricts to a surjective holomorphic map $M\longrightarrow
  X$, which shows $M\geq X$.  As $X$ has strictly negative sectional
  curvature, its simplicial volume is positive, and the relation
  $M\geq X$ shows that $M$ also has positive simplicial volume. Thus
  Corollary~\ref{c:amen} shows $P\ngeq M$ for any non-trivial product
  $P$, although $\pi_1(M)$ is a product.
\end{ex}
On the one hand, this example certainly shows that
Corollary~\ref{c:amen} is not empty, although it is conjectured that a
closed aspherical manifold whose fundamental group contains a
non-trivial amenable normal subgroup has vanishing simplicial volume.
The manifold $M$ in the example is rationally essential with non-zero
simplicial volume, but not aspherical. On the other hand, for this
particular $M$ the conclusion $P\ngeq M$ also follows from the
combination of $M\geq X$ and $P\ngeq X$, with the latter being a
consequence of Theorems~\ref{t:top} and~\ref{t:alg}.

\section{Outlook}\label{s:final}

In this section we discuss the relation between the domination
relation and the ranks of fundamental groups, and we also consider the
domination relation in general, without restricting to product domains
and to targets with large fundamental groups.

\subsection*{Relationship with the rank of groups}

Building on ideas of Tits, Prasad and Raghunathan, Ballmann and
Eberlein~\cite{BE} introduced an abstract notion of the rank of a
group $\Gamma$. This notion is a measure of the size of centralisers
of generic elements for all finite index subgroups of $\Gamma$. In
particular, the rank does not change under passage to a subgroup of
finite index.  It is a result of Prasad--Raghunathan~\cite{PR} that
the rank of a cocompact lattice in the isometry group of a symmetric
space of non-compact type coincides with the rank (in the sense of Lie
theory) of the symmetric space. Ballmann--Eberlein~\cite{BE} proved
that for the fundamental groups of closed manifolds of non-positive
sectional curvature the rank agrees with the geometric rank of the
Riemannian metric defined via spaces of parallel Jacobi fields (which
in turn agrees with the Lie theoretic rank in the case of locally
symmetric spaces).

The groups occurring in Theorem~\ref{t:alg} are all of rank one. In
the statement~(N-P) this is part of the assumption because, by the
result of Ballmann--Eberlein, the geometric and algebraic rank are the
same. For infinite hyperbolic groups it is well-known folklore that
they have rank one -- indeed one could argue that this is part of the
definition, or at least of the philosophy behind their
definition~\cite{gromovhyperbolic}. For mapping class groups the
determination of the rank is a theorem of Ivanov~\cite{IvanRank}.

It is possible that Theorem~\ref{t:alg} is a special case of a more
general result asserting that rank one groups in the sense of
Ballmann--Eberlein that are not virtually cyclic are not presentable
by products.  The proof of such a statement in the general case runs
into many algebraic technicalities, so we do not pursue it here.
Moreover, in concrete cases it has turned out to be easier to prove
directly that a group is not presentable by a product, instead of
proving that it has rank one.  A case in point is provided by the
mapping class groups, for which our proof of statement~(MCG) in
Theorem~\ref{t:alg} is easier than Ivanov's proof~\cite{IvanRank} that
the groups have rank one.

\subsection*{The domination relation in general}

In this paper we have shown that many manifolds with large fundamental
groups are not dominated by products.  In contrast with this, it could
be true that all manifolds with finite fundamental groups are
dominated by products.  Using the results of Duan--Wang~\cite{DW1,DW2}
this is easy to verify in dimension four:
\begin{prop}\label{p:4mfd}
  Every closed oriented connected four-manifold with finite
  fundamental group is dominated by the product of a torus with a
  suitable closed oriented connected surface.
\end{prop}
\begin{proof}
  Clearly it is enough to prove the statement for simply connected
  targets.

  Let $P$ and $N$ be closed oriented four-manifolds with intersection
  forms $Q_P$ and $Q_N$ respectively.  If $f\colon P\longrightarrow N$
  is a map of degree $d$, then $d\cdot Q_N$ is embedded into $Q_P$ by
  the pullback $H^*(f)$. (Here $d\cdot Q_N$ means that all entries are
  multiplied by $d$.) It is a result of Duan and
  Wang~\cite[Theorem~3]{DW2} that for simply connected targets $N$ a
  map of degree $d$ exists whenever the necessary algebraic condition
  $d\cdot Q_N\subset Q_P$ is satisfied.

  We apply this in the case where $P= S^1\times S^1 \times\Sigma_g$
  for some~$g \in \bN$. This $P$ has zero signature, so that the
  embedding of intersection forms exists whenever the second Betti
  number of $P$ is large enough compared with the positive and
  negative parts of~$Q_N$ (which may have different rank because $N$
  may have non-zero signature). If $Q_N$ is odd, we have to choose $d$
  even to find such an algebraic embedding, because $Q_P$ is even. It
  suffices to take $g$ large with respect to $b^{\pm}_2(N)$.
\end{proof}

Carlson and Toledo~\cite[p.~174]{CT} mentioned that in arbitrary
dimensions there is no easily described class of manifolds forming
maximal elements for the relation $\geq$ in the sense that for every
closed oriented $M$ there should be an $N$ from this class such that
$N\geq M$. In this direction we propose the following:
\begin{conj}\label{max}
  In every dimension $n\geq 2$, closed oriented hyperbolic manifolds
  represent a maximal class of homotopy types with respect to the
  domination relation $\geq$.
\end{conj}
This is trivially true for $n=2$, and, more interestingly, it is known
to be true for $n=3$ by a result of Brooks~\cite{Brooks}.  We can also
verify a strong form of the conjecture for four-manifolds with finite
fundamental groups:
\begin{prop}
  Let $M$ be a closed oriented connected four-manifold with finite
  fundamental group. Then every closed oriented hyperbolic manifold
  $N$ virtually dominates $M$, i.e., some finite covering space of~$N$
  dominates~$M$.
\end{prop}

\begin{proof}
  The proof is almost the same as the proof of
  Proposition~\ref{p:4mfd}. Like that proof, this one is a direct
  consequence of the result of Duan--Wang~\cite[Theorem~3]{DW2},
  because we may assume $M$ to be simply connected.

  Fix a simply connected closed oriented four-manifold $M$ and a
  closed oriented hyperbolic four-manifold $N$. We want to show that
  $N$ virtually dominates $M$. Recall that $N$ has positive Euler
  characteristic and zero signature~\cite{K}.  Moreover, $\pi_1(N)$ is
  a residually finite lattice, and so $N$ has finite coverings $\bar
  N$ of arbitrarily large degrees, equivalently of arbitrarily large
  Euler characteristics. This means that the second Betti numbers of
  these coverings are unbounded, and once they are large enough, we
  can verify the algebraic criterion $d\cdot Q_M\subset Q_{\bar N}$ of
  Duan--Wang~\cite[Theorem~3]{DW2}, say with $d=2$.
\end{proof}
Finally, in support of Conjecture~\ref{max}, we would like to mention
the following heuristic. A candidate class of maximal elements with
respect to $\geq$ should consist of manifolds that are not themselves
dominated by too many other manifolds, because otherwise the
transitivity of the relation might lead to contradictions. Hyperbolic
manifolds are indeed dominated by very few other manifolds: not by
non-trivial products (by our results), not by manifolds with zero
simplicial volume~\cite{Gromov}, not by K\"ahler
manifolds~\cite{CT,ABCKT}, and not even by certain K\"ahler--Weyl
manifolds~\cite{KK}.

\bibliographystyle{amsplain}

\begin{thebibliography}{10}

\bibitem{ABCKT}
J.~Amor\'os, M.~Burger, K.~Corlette, D.~Kotschick and D.~Toledo, 
{\sl Fundamental Groups of Compact K\"ahler Manifolds}, Mathematical Surveys and Monographs, Vol.~{\bf 44}, 
Amer.~Math.~Soc., Providence, R.I.~1996.

\bibitem{BBE}
W.~Ballmann, M.~Brin and P.~Eberlein, {\em Structure of manifolds of nonpositive curvature}, 
Ann.~of Math.~{\bf 122} (1985), 171--203.

\bibitem{BE}
W.~Ballmann and P.~Eberlein, {\em Fundamental groups of manifolds of nonpositive curvature}, 
J.~Differential Geometry {\bf 25} (1987), 1--22.

 \bibitem{bh}
 M.~R.~Bridson and A.~Haefliger, \textsl{Metric Spaces of Non-positive Curvature},
          Grundlehren der Mathematischen Wissenschaften Vol.~{\bf 319}, 
          Springer Verlag 1999.

\bibitem{Brooks}
R.~Brooks, {\em On branched coverings of $3$-manifolds which fiber over the circle},
J.~Reine Angew.~Math.~{\bf 362} (1985), 87--101. 

\bibitem{B}
M.~Brunnbauer, {\em Homological invariance for asymptotic invariants and systolic inequalities}, 
Geom.~Funct.~Analysis DOI 10.1007/s00039-008-0677-4 (online first, to appear in print).

\bibitem{CT}
J.~A.~Carlson and D.~Toledo, {\em Harmonic mapping of K\"ahler manifolds to locally symmetric spaces},
Publ.~Math.~I.H.E.S.~{\bf 69} (1989), 173--201.

\bibitem{DW1}
H.~Duan and S.~Wang, {\em The degrees of maps between manifolds}, Math.~Z.~{\bf 244} (2003), 67--89.

\bibitem{DW2}
H.~Duan and S.~Wang, {\em Non-zero degree maps between
  $2n$-manifolds}, Acta Math.\ Sin.\ (Engl.~Ser.) {\bf 20} (2004), 1--14.  

\bibitem{E}
P.~Eberlein, 
{\em A canonical form for compact nonpositively curved manifolds whose fundamental groups have nontrivial center},
Math.~Annalen {\bf 260} (1982), 23--29.

\bibitem{EbJDG}
P.~Eberlein, 
{\em Euclidean de Rham factor of a lattice of nonpositive curvature},
J.~Differential Geometry {\bf 18} (1983), 209--220.

\bibitem{FS}
S.~Fornari and V.~Schroeder, {\em Ramified coverings with nonpositive curvature}, 
Math.~Z.~{\bf 203} (1990), 123--128.

\bibitem{GW}
D.~Gromoll and J.~A.~Wolf, {\em Some relations between the metric structure and the algebraic structure 
of the fundamental group in manifolds of nonpositive curvature}, Bull.~Amer.~Math.~Soc.~{\bf 77} (1971), 545--552.

\bibitem{Gromov}
M.~Gromov, {\em Volume and bounded cohomology}, 
Publ.~Math.~I.H.E.S.~{\bf 56} (1982), 5--99.

\bibitem{gromovhyperbolic} 
M.~Gromov, \emph{Hyperbolic groups},
          in {\sl Essays in Group Theory} (ed.~S.~M.~Gersten),
           MSRI Publ.~\textbf{8}, Spinger Verlag 1987, pp.~75--263. 

\bibitem{gromovmetric} 
M.~Gromov, \textsl{Metric Structures for
          Riemannian and Non-Riemannian Spaces}, with appendices by
          M.~Katz, P.~Pansu and S.~Semmes, translated from the French
          by S.~M.~Bates, Progress in Mathematics Vol.~{\bf 152},
          Birkh\"auser Verlag 1999.

\bibitem{HKRS}
B.~Hanke, D.~Kotschick, J.~Roe and T.~Schick, 
{\em Coarse topology, enlargeability and essentialness}, 
Ann.~Sci.~Ecole Norm.~Sup.~{\bf 41} (2008), 471--493.

\bibitem{HS}
B.~Hanke and T.~Schick, 
{\em Enlargeability and index theory}, J.~Differential Geom.~{\bf 74} (2006), 293--320.

\bibitem{H}
J.~A.~Hillman, 
{\sl Four-manifolds, geometries and knots}, Geometry and Topology Monographs Vol.~{\bf 5} (2002).

 \bibitem{ivanov} 
 N.~V.~Ivanov, \emph{Foundations of the theory of
          bounded cohomology,} J.~Soviet Math.~\textbf{37} (1987),
          1090--1114.

 \bibitem{IvanRank} 
 N.~V.~Ivanov, {\em Rank of Teichm\"uller modular groups}, 
 (Russian) Mat.~Zametki {\bf 44} (1988), 636--644; English translation in 
Math.~Notes {\bf 44} (1988), 829--832.

 \bibitem{MCG} 
 N.~V.~Ivanov, {\sl Subgroups of Teichm\"uller Modular Groups},
 Translations of Math.~Monographs, Vol.~{\bf 115}, American Math.~Soc., Providence, R.I. 1992.
 
\bibitem{KL}
M.~Kapovich and B.~Leeb, {\em Actions of discrete groups on nonpositively curved spaces}, 
Math.~Annalen {\bf 306} (1996), 341--352.

\bibitem{KKM}
J.~K\c edra, D.~Kotschick and S.~Morita,
{\em Crossed flux homomorphisms and vanishing theorems for flux groups}, 
Geom.~Funct.~Analysis {\bf 16} (2006), 1246--1273.

\bibitem{KK}
G.~Kokarev and D.~Kotschick, {\em Fibrations and fundamental groups of K\"ahler--Weyl manifolds},
Preprint arXiv:0811.1952 v1 [math.DG] 12 Nov 2008.

\bibitem{K}
D.~Kotschick, {\em Remarks on geometric structures on compact complex 
surfaces}, Topology~{\bf 31} (1992), 317--321.

\bibitem{LY}
H.~B.~Lawson, Jr. and S.-T.~Yau, {\em Compact manifolds of nonpositive curvature}, 
J.~Differential Geom.~{\bf 7} (1972), 211--228.

\bibitem{lueck} W.~L\"uck, \emph{Survey on aspherical manifolds}, 
  preliminary version, 2008.

\bibitem{MT}
J.~W.~Milnor and W.~P.~Thurston,
{\em Characteristic numbers of $3$-manifolds},
Enseign.~Math.~{\bf 23} (1977), 249--254; reprinted in 
J.~W.~Milnor, {\sl Collected Papers}, Vol.~{\bf 2}, 
Publish or Perish 1995.

\bibitem{paterson} 
A.~L.~T.~Paterson, \textsl{Amenability}, 
Mathematical Surveys and Monographs, Vol.~{\bf 29}, Amer.~Math.~Soc., Providence, R.I.~1988.

\bibitem{PR}
G.~Prasad and M.~S.~Raghunathan, {\em Cartan subgroups and lattices in semisimple groups}, 
Ann.~of Math.~{\bf 96} (1972), 296--317.

 \bibitem{raghunathan} M.~S.~Raghunathan, \textsl{Discrete Subgroups of
          Lie Groups},  Vol.~{\bf 68} of \emph{Ergebnisse der Mathematik und
          ihrer Grenzgebiete}, Springer Verlag 1972.

\bibitem{Sam}
A.~Sambusetti, {\em Minimal entropy and simplicial volume}, Manuscripta math.~{\bf 99} (1999), 541--560. 

\bibitem{WangICM}
S.~Wang, {\em Non-zero degree maps between $3$-manifolds}, Proc.~of the ICM Beijing 2002 Vol.~{\bf II} 457--468,
Higher Education Press Beijing 2002.

 \bibitem{wittemorris} D.~Witte Morris, \textsl{Introduction to
          Arithmetic Groups}. Preliminary book, 2001. Available at 
          arXiv:math/0106063v3 [math.DG]

\end{thebibliography}

\bigskip

\end{document}